\documentclass[11pt, letterpaper]{article}

\usepackage[utf8]{inputenc}
\usepackage{lmodern}
\usepackage[T2A,T1]{fontenc}
\usepackage[russian, english]{babel}

\usepackage[margin=1in]{geometry}
\usepackage{amsmath, amsthm, amssymb, amsfonts}
\usepackage{mathtools}
\usepackage{bm}
\usepackage[sc]{mathpazo}
\usepackage{booktabs}
\usepackage{enumitem}
\usepackage{microtype}
\usepackage[dvipsnames]{xcolor}
\usepackage[nocompress]{cite}
\usepackage{xparse}

\usepackage{todonotes}
\usepackage{thmtools}
\usepackage[hypertexnames=false]{hyperref}
\usepackage{cleveref}

\hypersetup{
colorlinks=true,
linkcolor=NavyBlue,
citecolor=ForestGreen,
urlcolor=NavyBlue,
pdftitle={Explicit Universal Bounds for Cumulants via Moments},
pdfauthor={Jiechen Zhang},
}

\Crefname{theorem}{Theorem}{Theorems}
\Crefname{lemma}{Lemma}{Lemmas}
\Crefname{corollary}{Corollary}{Corollaries}
\Crefname{proposition}{Proposition}{Propositions}
\Crefname{definition}{Definition}{Definitions}
\Crefname{remark}{Remark}{Remarks}
\Crefname{section}{Section}{Sections}
\Crefname{equation}{Equation}{Equations}

\declaretheorem[name=Theorem, numberwithin=section]{theorem}
\declaretheorem[name=Lemma, sibling=theorem]{lemma}
\declaretheorem[name=Corollary, sibling=theorem]{corollary}
\declaretheorem[name=Proposition, sibling=theorem]{proposition}
\declaretheorem[name=Definition, style=definition, sibling=theorem]{definition}
\declaretheorem[name=Example, sibling=theorem]{example}
\declaretheorem[name=Remark, style=remark, sibling=theorem]{remark}

\DeclarePairedDelimiter{\abs}{\lvert}{\rvert}
\DeclarePairedDelimiter{\norm}{\lVert}{\rVert}

\DeclarePairedDelimiterX{\PParens}[1]{(}{)}{#1}
\DeclarePairedDelimiterX{\EBrackets}[1]{[}{]}{#1}

\RenewDocumentCommand{\Pr}{g}{
  \mathsf{P}\IfValueT{#1}{\PParens{#1}}
}
\newcommand{\E}{\mathsf{E}\EBrackets}
\newcommand{\Var}{\mathsf{Var}\EBrackets}

\newcommand{\R}{\mathbb{R}}

\newcommand{\Nzero}{\mathbb{N}_0}
\newcommand{\PP}{\mathcal{P}}
\newcommand{\stirlingii}[2]{\genfrac\{\}{0pt}{}{#1}{#2}}

\newcommand{\Craw}[1]{C^{\mathrm{raw}}_{#1}}         
\newcommand{\Ccen}[1]{C^{\mathrm{cen}}_{#1}}         
\newcommand{\Csym}[1]{C^{\mathrm{sym}}_{#1}}

\title{\bfseries Explicit Universal Bounds for Cumulants via Moments}
\author{Jiechen Zhang\footnote{EPFL. Email: \href{mailto:jiechen.zhang@epfl.ch}{jiechen.zhang@epfl.ch}.}}
\date{}

\begin{document}

\maketitle

\begin{abstract}
We establish explicit, universal, and distribution-free bounds for the $n$-th cumulant, $\kappa_n(X)$, of a scalar random variable, controlled solely by an $n$-th order absolute moment functional $M_n(X)$. The bounds take the form $\lvert\kappa_n(X)\rvert \le C_n M_n(X)$. Our principal contribution is the derivation of coefficients satisfying $C_n \sim (n-1)!/\rho^{\,n}$, which offers an exponential improvement over classical bounds where the coefficients grow superexponentially (on the order of $n^n$).

We present a hierarchy of refinements where the rate parameter $\rho$ increases as the functional $M_n(X)$ incorporates more structural information. The most general bound uses the raw moment $M_n(X)=\mathsf{E}[\lvert X\rvert^n]$ with rate $\rho=\ln 2 \approx 0.693$. Using the central moment $M_n(X)=\mathsf{E}[\lvert X-\mathsf{E}[X]\rvert^n]$ improves the rate to $\rho_{\mathrm{cen}} \approx 1.146$, while assuming symmetry yields even higher rates.

The proof is elementary, combining the moment-cumulant partition formula with a uniform moment-product inequality. We further prove that while these bounds are not attainable whenever the relevant coefficient is positive, they are asymptotically efficient given the limited information of a single moment. The utility of the bounds is demonstrated through an application to standardized cumulants of independent sums.
\end{abstract}

\section{Introduction}

Cumulants (or semi-invariants) are fundamental descriptors of the statistical properties of a random variable. Unlike moments, which mix information of various orders, the $n$-th cumulant $\kappa_n(X)$ isolates the pure $n$-th order correlation structure, representing the part of the moment that cannot be explained by lower-order statistics. This property makes them indispensable in the study of independence, non-Gaussianity, and limit theorems. They appear as the coefficients in the logarithm of the characteristic function, serve as the controlling parameters in Edgeworth expansions, and play a central role in concentration inequalities, random matrix theory, and Stein's method.

Despite their theoretical importance, the practical manipulation of high-order cumulants is somehow difficult. The defining relationship between moments and cumulants involves an alternating sum over partitions, leading to significant algebraic cancellations. Consequently, while moments $\E{\abs{X}^n}$ typically grow regularly, cumulants can exhibit erratic behavior or vanish entirely (as in the Gaussian case for $n \ge 3$). Bounding cumulants is therefore a frequent necessity, yet existing methods often require strong distributional assumptions or yield bounds that are too coarse for high-order analysis.

\paragraph{Related Work.}
Existing strategies for bounding cumulants typically fall into three categories. 

\emph{Analytic methods} leverage the properties of a random variable's generating function. For instance, assuming the moment generating function is analytic in a strip around the origin allows for Cauchy-type integral estimates. This approach leads to sharp factorial bounds, often referred to as Statulevi\v{c}ius-type conditions \cite{statulevivcius1966large, doring2022method}. However, these techniques are inapplicable to heavy-tailed distributions where the generating function diverges or is not analytic.

\emph{Structural methods} exploit specific properties of the underlying model, such as the dependence graph in graphical models, mixing coefficients in time series, or chaos decompositions for functionals of Gaussian processes. In a complementary direction, Utev \cite{utev1989cumulants} used cumulant estimates to derive moment inequalities for sums of weakly dependent random variables under mixing and related dependence assumptions. While powerful, these approaches require structural assumptions (e.g., specific decay of correlations) that may not hold in general settings or can be difficult to verify.

\emph{Moment-based bounds}, in contrast, rely solely on the finiteness of moments and are therefore universally applicable. However, they are historically plagued by coarseness. While best-possible constants have been achieved for Rosenthal-type moment inequalities and sharp tail bounds for sums and martingales \cite{pinelis1994optimum, pinelis2015rosenthal,  pinelis2007asymmetric}, similar explicit sharpness has been lacking for the cumulants of individual variables. However, exact cumulant analysis remains a powerful tool for establishing non-conventional limit theorems, as seen in recent work on integer partitions\cite{stoyanov2020nonconventional}.  While bounds of the form $\abs{\kappa_n(X)} \le C_n \E{\abs{X-\E{X}}^n}$ are known, their utility is limited by the superexponential growth of the coefficient, typically on the order of $C_n \approx n^n$ \cite{ProkhorovRozanov1969_en, dubkov1976properties, pinelis2020bounds, jones2023bound}. This rapid growth renders them loose and often vacuous for large $n$.

\paragraph{Our Contribution.}
In this work, we present a unified framework that bridges the gap between the universality of moment-based bounds and the sharpness of analytic estimates. We establish explicit, distribution-free inequalities for the $n$-th cumulant controlled solely by the $n$-th absolute moment, yet we achieve a factorial growth rate $C_n \sim (n-1)!$ rather than the superexponential $n^n$.

We prove that for any scalar random variable $X$,
\[
\abs{\kappa_n(X)} \le C_n M_n(X),
\]
where the coefficient $C_n$ is the exact "combinatorial mass" of the partition formula associated with the structural properties of $X$ and $M_n(X)$ is an $n$-th order absolute moment functional. By combining the classical moment-cumulant formula \cite{leonov1959method, mccullagh1987tensor} with a uniform moment-product inequality derived from Lyapunov's inequality, we identify the precise constants $C_n$ and analyze their asymptotics.

Our results are organized as a hierarchy of refinements, where the bound tightens as we incorporate minimal information about the centering or symmetry of the variable:
\begin{enumerate}[label=(\roman*), topsep=2pt, itemsep=0pt, leftmargin=*]
    \item \textbf{Raw moment bound.} For every $n\ge1$,
    \[
    |\kappa_n(X)|\le \Craw{n}\,\E{|X|^n},
    \qquad
    \Craw{n}=2\operatorname{Bell}(n-1)\ \ (n\ge2),
    \]
    and $\Craw{n}\sim (n-1)!/(\ln 2)^n$.

    \item \textbf{Centered bound.} For every $n\ge2$,
    $|\kappa_n(X)|\le \Ccen{n}\,\E{|X-\E{X}|^n},$
    where $\Ccen{n}$ is the total coefficient mass over partitions with no singleton blocks, and
    \[
    \Ccen{n}\sim \frac{(n-1)!}{\rho_{\mathrm{cen}}^n},
    \qquad
    e^{\rho_{\mathrm{cen}}}=2+\rho_{\mathrm{cen}}.
    \]

    \item \textbf{Symmetric bound.} If $X\overset d= -X$, then $\kappa_n(X)=0$ for odd $n$, while for even $n=2m$,
    $|\kappa_{2m}(X)|\le \Csym{2m}\,\E{|X|^{2m}},$
    where $\Csym{2m}$ is the total coefficient mass over partitions with only even-sized blocks, and
    \[
    \Csym{2m}\sim 2\frac{(2m-1)!}{\rho_{\mathrm{sym}}^{2m}},
    \qquad
    \rho_{\mathrm{sym}}=\operatorname{arcosh}(2).
    \]
\end{enumerate}

The same method extends to multivariate joint cumulants with the same coefficients. We also prove non-attainment whenever the corresponding coefficient is positive, apart from the identity in the centered variance case $n=2$; for symmetric odd orders the coefficient is zero and the cumulant vanishes identically. Finally, we illustrate the bounds through an application to standardized cumulants of independent sums.

These bounds are assumption-free beyond the existence of the moment itself. While our work focuses on what a single moment can reveal about the corresponding cumulant, the fundamental question of how the entire sequence of moments characterizes a distribution is the subject of the classical moment problem. Extensive work has been done to determine conditions for moment determinacy based on the asymptotic growth of moments \cite{kleiber2013multivariate,lin2002moment,stoyanov2013counterexamples,stoyanov2020new}. Here, rather than aiming at global moment determinacy, we use the combinatorial structure of the moment--cumulant formula to derive explicit bounds on individual cumulants using only a single absolute moment.

\paragraph{Organization.}
\Cref{sec:prelim} recalls the moment--cumulant formula and the basic product estimate. \Cref{sec:main_bounds} states the raw, centered, and symmetric bounds. \Cref{sec:coefficient_analysis} analyzes the corresponding coefficients by generating functions. \Cref{sec:converse} discusses the absence of any reverse inequality depending only on a single moment. The multivariate extension appears in \Cref{sec:multivariate}. \Cref{sec:strict} proves non-attainment and compares the symmetric rate with a unit-circle benchmark. \Cref{sec:applications} contains an application to standardized cumulants of independent sums.

\section{Technical Preliminaries}
\label{sec:prelim}

\subsection{Cumulants and the Partition Formula}

For a random variable $X$ with $\E{\abs{X}^n} < \infty$, all moments $m_k = \E{X^k}$ for $k \le n$ exist. The cumulants $\kappa_1, \dots, \kappa_n$ are well-defined and can be computed directly from the moments.

We use the algebraic definition of cumulants through the moment--cumulant formula, which requires only the finiteness of moments up to the relevant order.

\begin{lemma}
\label{lem:mom-cum}
Let $n \ge 1$ and assume $\E{\abs{X}^n} < \infty$. The $n$-th cumulant $\kappa_n(X)$ is given by
\begin{equation}
\label{eq:mom-cum-formula}
\kappa_n(X) = \sum_{\pi \in \PP(n)} (-1)^{\abs{\pi}-1} (\abs{\pi}-1)! \prod_{B \in \pi} m_{\abs{B}},
\end{equation}
where $\PP(n)$ is the set of all partitions of $[n]\coloneqq\{1,\dots,n\}$, $\abs{\pi}$ is the number of blocks in a partition $\pi$, and the product is over all blocks $B$ in $\pi$. Here $m_k$ denotes the $k$-th raw moment.
\end{lemma}
This identity is a classical result in combinatorics and statistics. It is the univariate instance of the general moment-cumulant formula for joint cumulants (or semi-invariants) detailed in the foundational work of Leonov and Shiryaev \cite{leonov1959method}. For modern treatments that derive this formula via Möbius inversion on the lattice of set partitions, see \cite{speed1983cumulants} and the textbook by \cite{mccullagh1987tensor}.

\begin{lemma}[Basic identities for cumulants]
\label{lem:cumulant-identities}
Let $X,Y$ be real-valued random variables with finite moments up to order $n$.

\begin{enumerate}[label=(\roman*), leftmargin=*]
    \item For every scalar $a\in\R$ and every integer $r\ge 1$,
    $\kappa_r(aX)=a^r\kappa_r(X).$
    \item If $X$ and $Y$ are independent, then for every integer $r\ge 1$,
    $\kappa_r(X+Y)=\kappa_r(X)+\kappa_r(Y).$
    \item For every constant $c\in\R$,
    $\kappa_1(X+c)=\kappa_1(X)+c,
    \qquad
    \kappa_r(X+c)=\kappa_r(X)\quad (r\ge 2).$
\end{enumerate}
\end{lemma}

\begin{proof}
Part (i) follows from \Cref{eq:mom-cum-formula}: the raw moments of $aX$ satisfy
$m_k(aX)=a^k m_k(X)$, so each summand in the partition formula acquires the factor
$\prod_{B\in\pi} a^{|B|}=a^r$.

For part (ii), fix $r\ge1$ and work in the quotient ring $\R[[t]]/(t^{r+1})$. Define the truncated moment series
\[
\widetilde M_X^{(r)}(t)\coloneqq \sum_{j=0}^{r} m_j(X)\frac{t^j}{j!},
\qquad
\widetilde M_Y^{(r)}(t)\coloneqq \sum_{j=0}^{r} m_j(Y)\frac{t^j}{j!}.
\]
Since $X$ and $Y$ are independent, the moments of $X+Y$ satisfy
\[
m_j(X+Y)=\sum_{\ell=0}^{j}\binom{j}{\ell}m_{\ell}(X)m_{j-\ell}(Y)
\qquad (0\le j\le r),
\]
hence
\[
\widetilde M_{X+Y}^{(r)}(t)=\widetilde M_X^{(r)}(t)\,\widetilde M_Y^{(r)}(t)
\qquad\text{in }\R[[t]]/(t^{r+1}).
\]
Each truncated moment series has constant term $1$, so the formal logarithm is well defined. By definition of cumulants,
\[
\log \widetilde M_X^{(r)}(t)=\sum_{j=1}^{r}\kappa_j(X)\frac{t^j}{j!},
\qquad
\log \widetilde M_Y^{(r)}(t)=\sum_{j=1}^{r}\kappa_j(Y)\frac{t^j}{j!},
\]
and similarly for $X+Y$. Therefore
\[
\log \widetilde M_{X+Y}^{(r)}(t)
=
\log \widetilde M_X^{(r)}(t)+\log \widetilde M_Y^{(r)}(t)
\qquad\text{in }\R[[t]]/(t^{r+1}),
\]
and comparing coefficients of $t^r/r!$ gives
$\kappa_r(X+Y)=\kappa_r(X)+\kappa_r(Y).$

For part (iii), apply part (ii) with the degenerate variable $c$. Its moment series is
$\widetilde M_c^{(r)}(t)=\sum_{j=0}^{r} c^j\frac{t^j}{j!},$
whose formal logarithm is simply $ct$ modulo $t^{r+1}$. Thus $\kappa_1(c)=c$ and
$\kappa_r(c)=0$ for $r\ge2$, which yields the stated shift identities.
\end{proof}

\subsection{A Uniform Moment Product Bound via Lyapunov's Inequality}

The second key ingredient is a powerful consequence of Lyapunov's inequality that allows us to uniformly bound any product of moments appearing in \Cref{lem:mom-cum}.

\begin{lemma}
\label{lem:product_collapse}
Let $n \ge 1$ and assume $\mu_n = \E{\abs{X}^n} < \infty$. For any partition $\pi \in \PP(n)$,
$\prod_{B \in \pi} \mu_{\abs{B}} \le \mu_n.$
\end{lemma}
\begin{proof}
The proof is a direct application of Lyapunov's inequality (monotonicity of $L^p$ norms on a probability space), which establishes the log-convexity of absolute moments. For completeness, a proof is provided in \Cref{rem:lyapunov_proof}. Applying Lyapunov's inequality to each block size $\abs{B}$ in the partition $\pi$ gives $(\mu_j)^{1/j} \le (\mu_n)^{1/n}$, or $\mu_j \le \mu_n^{j/n}$ for $j<n$. Thus,
$\mu_{\abs{B}} \le \mu_n^{\abs{B}/n}.$
Multiplying over all blocks $B \in \pi$ yields
\[
\prod_{B \in \pi} \mu_{\abs{B}} \le \prod_{B \in \pi} \mu_n^{\abs{B}/n} = \mu_n^{\sum_{B \in \pi} \abs{B}/n} = \mu_n^{n/n} = \mu_n,
\]
since the sum of the sizes of the blocks in a partition of $\{1, \dots, n\}$ is $n$.
\end{proof}

\begin{remark}[Proof of Lyapunov's Inequality]
\label{rem:lyapunov_proof}
Let $0 < j < k$. We want to show $(\mu_j)^{1/j} \le (\mu_k)^{1/k}$. Let $Y = \abs{X}^j$ and $p = k/j > 1$. The function $\phi(z) = z^p$ on $[0,\infty)$ is convex. By Jensen's inequality:
$\phi(\E{Y}) \le \E{\phi(Y)}.$
Substituting our definitions gives:
$(\E{\abs{X}^j})^{k/j} \le \E{(\abs{X}^j)^{k/j}} = \E{\abs{X}^k}.$
Taking the $k$-th root of both sides yields $(\mu_j)^{1/j} \le (\mu_k)^{1/k}$.
\end{remark}

\begin{lemma}[Equality in Lyapunov on a probability space]
\label{lem:lyapunov-equality}
Let $Y\ge0$ be measurable on a probability space and let $0<p<q<\infty$. Then
$\norm{Y}_{p}\le \norm{Y}_{q}.$
If $\norm{Y}_{q}<\infty$, equality holds if and only if $Y$ is almost surely constant.
\end{lemma}

\begin{proof}
The inequality is Lyapunov's inequality. For the equality statement, write $r=q/p>1$ and apply Jensen to the convex function $z\mapsto z^r$ on $[0,\infty)$ with the random variable $Y^p$:
$(\E{Y^p})^{r}\le \E{Y^{pr}}=\E{Y^q}.$
Equality in Jensen for a strictly convex function occurs if and only if $Y^p$ is almost surely constant, hence if and only if $Y$ is almost surely constant.
\end{proof}

\subsection{Stirling and Ordered Bell Numbers}

We use $\stirlingii{n}{k}$ for Stirling numbers of the second kind (the number of partitions of an $n$-element set into $k$ nonempty blocks).

\begin{definition}[Ordered Bell (Fubini) Numbers]
\label{def:ordered_bell}
The ordered Bell number $\operatorname{Bell}(m)$ counts the number of ordered partitions of an $m$-element set:
$\operatorname{Bell}(m) \coloneqq \sum_{k=0}^{m} \stirlingii{m}{k} k!.$
\end{definition}

\section{Main Results: Universal Bounds for Cumulants}
\label{sec:main_bounds}

Our main results provide universal bounds on the $n$-th cumulant using only its corresponding absolute moment. We begin by defining our notation and then present the main bounds.

\begin{definition}
Let $X$ be a real-valued random variable. For any $r>0$ such that $\E{\abs{X}^r} < \infty$, we define the \emph{$r$-th absolute moment} as $\mu_r \coloneqq \E{\abs{X}^r}$. For any integer $k \ge 1$ such that $\mu_k < \infty$, we define the \emph{$k$-th raw moment} as $m_k \coloneqq \E{X^k}$. Note that since $|\cdot|$ is a convex function, by Jensen's inequality, $|m_k| = |\E{X^k}| \le \E{|X^k|} = \E{\abs{X}^k} = \mu_k$.
\end{definition}

\subsection{The General Raw-Moment Bound}

\begin{theorem}
\label{thm:main}
Let $X$ be a real-valued random variable. For any integer $n \ge 1$, if $\mu_n = \E{\abs{X}^n} < \infty$, then the $n$-th cumulant $\kappa_n(X)$ satisfies
$\abs{\kappa_n(X)} \le \Craw{n} \mu_n.$
For $n \ge 2$, this coefficient is exactly $2 \operatorname{Bell}(n-1)$, where $\operatorname{Bell}(m)$ are the ordered Bell numbers.
\end{theorem}

\begin{proof}
The proof relies on taking the absolute value of the moment--cumulant formula (\Cref{lem:mom-cum}).
\begin{align*}
\abs{\kappa_n(X)} &= \abs*{ \sum_{\pi \in \PP(n)} (-1)^{\abs{\pi}-1} (\abs{\pi}-1)! \prod_{B \in \pi} m_{\abs{B}} } &&\text{(By \Cref{lem:mom-cum})} \\
&\le \sum_{\pi \in \PP(n)} (\abs{\pi}-1)! \prod_{B \in \pi} \abs{m_{\abs{B}}} &&\text{(Triangle inequality)} \\
&\le \sum_{\pi \in \PP(n)} (\abs{\pi}-1)! \prod_{B \in \pi} \mu_{\abs{B}} &&\text{(Since $\abs{m_k} \le \mu_k$)} \\
&\le \left( \sum_{\pi \in \PP(n)} (\abs{\pi}-1)! \right) \mu_n &&\text{(By \Cref{lem:product_collapse})} \\
&= 2\operatorname{Bell}(n-1) \mu_n. &&\text{(By \Cref{lem:coefficient_mass}, for $n \ge 2$)}\\
&\coloneqq \Craw{n} \mu_n.
\end{align*}
\end{proof}

\begin{remark} [Case: $n = 1$]
    For the trivial case $n=1$, the sum is over the single partition $\pi=\{\{1\}\}$, giving a coefficient of $(|\pi|-1)! = (1-1)! = 1$. The bound is $|\kappa_1| = |m_1| \le \mu_1$, which is sharp. The identity relating the coefficient sum to ordered Bell numbers, $\sum_{\pi \in \PP(n)} (|\pi|-1)! = 2\operatorname{Bell}(n-1)$, holds for $n \ge 2$. While one can define $\operatorname{Bell}(0)=1$, applying the formula for $n=1$ yields a coefficient of $2\operatorname{Bell}(0)=2$, which provides a valid but non-sharp bound. Therefore, for convenience, we define $\Craw{1} = 1$.
\end{remark}

\begin{corollary}[Centered Form (basic)]
\label{cor:centered-basic}
For any $n \ge 2$, if $\E{\abs{X}^n} < \infty$, then
$\abs{\kappa_n(X)} = \abs{\kappa_n(X - \E{X})} \le \Craw{n} \E{\abs{X - \E{X}}^n}.$
This is a direct application of \Cref{thm:main} but is generally looser than the refined bound in \Cref{thm:central_sharp}.
\end{corollary}

\begin{proof}
By \Cref{lem:cumulant-identities}(iii), $\kappa_n(X)=\kappa_n(X-\E{X})$ for $n\ge2$. Also,
\[
\abs{X-\E{X}}^n\le 2^{n-1}\bigl(\abs{X}^n+\abs{\E{X}}^n\bigr)
\le 2^{n-1}\bigl(\abs{X}^n+\E{\abs{X}^n}\bigr),
\]
so $\E{\abs{X-\E{X}}^n}<\infty$. Applying \Cref{thm:main} to $X-\E{X}$ gives the claim.
\end{proof}

\begin{remark}[Asymptotic Improvement]
As we show in \Cref{sec:coefficient_analysis}, the coefficient has the asymptotic behavior $\Craw{n} = 2\operatorname{Bell}(n-1) \sim \frac{(n-1)!}{(\ln 2)^n}$. This is exponentially smaller in $n$ than the coefficient mass of the form $n^n$ that appears in other general-purpose bounds \cite{ProkhorovRozanov1969_en, dubkov1976properties, pinelis2020bounds, jones2023bound}.
\end{remark}

\subsection{The Universal Central-Moment Bound}

For orders $n\ge 2$, cumulants are shift invariant: $\kappa_n(X)=\kappa_n(X-\E{X})$. Consequently, any bound stated for centered variables is automatically a bound for any $X$, provided one evaluates moments at the centered variable. We never require the structural assumption $\E{X}=0$; we only use this shift-invariance. This leads to our headline result.

\begin{theorem}
\label{thm:central_sharp}
Let $X$ be a real-valued random variable, let $n\ge2$, and write
$\mu_k^{(c)}\coloneqq \E{\abs{X-\E{X}}^k}
\qquad (1\le k\le n).$
Assume $\mu_n^{(c)}<\infty$. Then
$\abs{\kappa_n(X)} \le \Ccen{n}\,\mu_n^{(c)},$
where
\[
\Ccen{n}\coloneqq
\sum_{\pi\in\PP_{\ge2}(n)}(\abs{\pi}-1)!,
\qquad
\PP_{\ge2}(n)\coloneqq \{\pi\in\PP(n): |B|\ge2 \text{ for every } B\in\pi\}.
\]
Moreover,
$\Ccen{n} \sim \frac{(n-1)!}{\rho_{\mathrm{cen}}^n}
\qquad (n\to\infty),$
where $\rho_{\mathrm{cen}}>0$ is the unique positive solution of
$e^{\rho_{\mathrm{cen}}}=2+\rho_{\mathrm{cen}}$.
\end{theorem}

\begin{proof}
Let $Y\coloneqq X-\E{X}$. By \Cref{lem:cumulant-identities}(iii),
$\kappa_n(X)=\kappa_n(Y).$
Since $m_1(Y)=\E{Y}=0$, every partition in \eqref{eq:mom-cum-formula} containing a singleton block contributes zero. Therefore
\[
\kappa_n(Y)
=
\sum_{\pi\in\PP_{\ge2}(n)}
(-1)^{|\pi|-1}(|\pi|-1)!\prod_{B\in\pi} m_{|B|}(Y).
\]
Taking absolute values and using $|m_k(Y)|\le \mu_k^{(c)}$ for every $k\le n$, we obtain
\[
|\kappa_n(X)|
=
|\kappa_n(Y)|
\le
\sum_{\pi\in\PP_{\ge2}(n)}
(|\pi|-1)!\prod_{B\in\pi} \mu_{|B|}^{(c)}.
\]
Applying \Cref{lem:product_collapse} to the centered variable $Y$ yields
$\prod_{B\in\pi}\mu_{|B|}^{(c)}\le \mu_n^{(c)}
\qquad (\pi\in\PP_{\ge2}(n)),$
and hence
\[
|\kappa_n(X)|
\le
\left(\sum_{\pi\in\PP_{\ge2}(n)}(|\pi|-1)!\right)\mu_n^{(c)}
=
\Ccen{n}\,\mu_n^{(c)}.
\]
The stated asymptotic is proved in \Cref{prop:C0-asymptotics}.
\end{proof}

\begin{corollary}
\label{cor:centered-refinement}
If $X$ is centered ($\E{X} = 0$) and $\E{\abs{X}^n}<\infty$, then for $n \ge 2$
$\abs{\kappa_{n}(X)} \le \Ccen{n} \E{\abs{X}^{n}}.$
\end{corollary}

\begin{proof}
Since $\E{X} = 0$, we have
$\E{\abs{X - \E{X}}^{n}} = \E{\abs{X - 0}^{n}} = \E{\abs{X}^{n}}.$
The result follows by applying \Cref{thm:central_sharp} with $\mu_n^{(c)} = \E{\abs{X - \E{X}}^n} = \E{\abs{X}^n}$.
\end{proof}

\subsection{Refinement for Symmetric Distributions}
\label{sec:refinements}

The universal bound can be further sharpened if the random variable has a symmetric distribution. In this case, all odd moments vanish, which eliminates any partition containing a block of odd size from the moment-cumulant formula.

If $X$ has a symmetric distribution ($X \overset{d}= -X$), then $m_{2r+1}=0$ for all $r\ge 0$. Let $\Csym{n}$ denote the total coefficient mass over partitions whose blocks all have even size:
\[
\sum_{\substack{\pi\in\PP(n)\\ \forall B\in\pi:\ |B|\text{ even}}} (\abs{\pi}-1)! = \Csym{n}.
\]
The combinatorial structure of $\Csym{n}$ and its generating function and asymptotic analysis are detailed in \Cref{subsec:Csym-analysis}.

\begin{theorem}
\label{thm:symmetric-refinement}
Assume that $X\overset d= -X$ and $\E{\abs{X}^n}<\infty$.

\begin{enumerate}[label=(\roman*), leftmargin=*]
    \item If $n$ is odd, then $\kappa_n(X)=0$.
    \item If $n=2m$ is even, then
    $\abs{\kappa_{2m}(X)} \le \Csym{2m}\,\E{\abs{X}^{2m}},$
    where
    \[
    \Csym{2m}\coloneqq
    \sum_{\pi\in\PP_{\mathrm{even}}(2m)}(\abs{\pi}-1)!,
    \qquad
    \PP_{\mathrm{even}}(2m)\coloneqq
    \{\pi\in\PP(2m): |B| \text{ is even for every } B\in\pi\}.
    \]
    Moreover,
    $\Csym{2m}\sim 2\frac{(2m-1)!}{\rho_{\mathrm{sym}}^{2m}}
    \qquad (m\to\infty),$
    where $\rho_{\mathrm{sym}}=\operatorname{arcosh}2$.
\end{enumerate}
\end{theorem}

\begin{proof}
If $X\overset d= -X$, then $m_{2r+1}=0$ for every $r\ge0$.

If $n$ is odd, every partition of $[n]$ has at least one block of odd cardinality, so each term in
\eqref{eq:mom-cum-formula} contains an odd moment and therefore vanishes. Hence $\kappa_n(X)=0$.

Now let $n=2m$. The same observation shows that only partitions whose blocks all have even size survive in the moment--cumulant formula:
\[
\kappa_{2m}(X)
=
\sum_{\pi\in\PP_{\mathrm{even}}(2m)}
(-1)^{|\pi|-1}(|\pi|-1)!\prod_{B\in\pi} m_{|B|}.
\]
Therefore
\[
|\kappa_{2m}(X)|
\le
\sum_{\pi\in\PP_{\mathrm{even}}(2m)}
(|\pi|-1)!\prod_{B\in\pi}|m_{|B|}|
\le
\sum_{\pi\in\PP_{\mathrm{even}}(2m)}
(|\pi|-1)!\prod_{B\in\pi}\mu_{|B|},
\]
where $\mu_k=\E{|X|^k}$. Applying \Cref{lem:product_collapse} gives
\[
\prod_{B\in\pi}\mu_{|B|}\le \mu_{2m}
\qquad (\pi\in\PP_{\mathrm{even}}(2m)),
\]
hence
\[
|\kappa_{2m}(X)|
\le
\left(\sum_{\pi\in\PP_{\mathrm{even}}(2m)}(|\pi|-1)!\right)\mu_{2m}
=
\Csym{2m}\,\E{|X|^{2m}}.
\]
The stated asymptotic follows from \Cref{prop:Csym-asymptotics}.
\end{proof}

\begin{table}[h!]
\centering
\caption{Summary of Asymptotic Coefficients and Parameters.}
\label{tab:summary}
\begin{tabular}{@{}lllc@{}}
\toprule
\textbf{Case} & \textbf{Assumption} & \textbf{Asymptotic Form of $C_n$} & \textbf{Parameter $\rho$} \\ \midrule
General & None & $\Craw{n} \sim (n-1)! \,/\, (\ln 2)^n$ & $0.693$ \\[1ex]
Centered & $n \ge 2$ & $\Ccen{n} \sim (n-1)! \,/\, \rho_{\mathrm{cen}}^n$ & $1.146$ \\[1ex]
Symmetric & $X \overset{d}= -X$ & $\Csym{2m} \sim 2(2m-1)! \,/\, \rho_{\mathrm{sym}}^{2m}$ & $1.317$ \\ \bottomrule
\end{tabular}
\end{table}

\begin{table}[h!]
\centering
\caption{Comparison of explicit coefficients for $n=2$ to $9$. Note the rapid divergence of the raw bound $\Craw{n}$ compared to the refinements $\Ccen{n}$ (no singletons) and $\Csym{n}$ (even blocks only).}
\label{tab:explicit_values}
\begin{tabular}{@{}c|r r r r r r r r@{}}
\toprule
$n$ & 2 & 3 & 4 & 5 & 6 & 7 & 8 & 9 \\ \midrule
$\Craw{n}$ & 2 & 6 & 26 & 150 & 1\,082 & 9\,366 & 94\,586 & 1\,091\,670 \\
$\Ccen{n}$ & 1 & 1 & 4 & 11 & 56 & 267 & 1\,730 & 11\,643 \\
$\Csym{n}$ & 1 & 0 & 4 & 0 & 46 & 0 & 1\,114 & 0 \\ \bottomrule
\end{tabular}
\end{table}

\section{Analysis of the Combinatorial Coefficients}
\label{sec:coefficient_analysis}

The coefficients in our bounds arise from summing $(\abs{\pi}-1)!$ over different families of partitions. We analyze them using exponential generating functions (EGFs).

\subsection{The General Case: Ordered Bell Numbers}

The coefficient $2\operatorname{Bell}(n-1)$ from \Cref{thm:main} arises from summing over all partitions.

\begin{lemma}
\label{lem:coefficient_mass}
For $n \ge 2$, the total mass of the coefficients in the moment--cumulant formula is
\[
\sum_{\pi \in \PP(n)} (\abs{\pi}-1)! = \sum_{k=1}^{n} \stirlingii{n}{k} (k-1)! = 2\operatorname{Bell}(n-1).
\]
\end{lemma}
\begin{proof}
The first equality follows from grouping partitions by their number of blocks, $k = \abs{\pi}$. There are $\stirlingii{n}{k}$ partitions with $k$ blocks. To prove the second equality, we use the standard recurrence for Stirling numbers of the second kind, $\stirlingii{n}{k} = \stirlingii{n-1}{k-1} + k \stirlingii{n-1}{k}$. Let $S_n = \sum_{k=1}^{n} \stirlingii{n}{k} (k-1)!$. For $n \ge 2$:
\begin{align*}
S_n &= \sum_{k=1}^{n} \left(\stirlingii{n-1}{k-1} + k \stirlingii{n-1}{k}\right) (k-1)! \\
&= \sum_{k=1}^{n} \stirlingii{n-1}{k-1} (k-1)! + \sum_{k=1}^{n} k! \stirlingii{n-1}{k} \\
&= \sum_{j=0}^{n-1} \stirlingii{n-1}{j} j! + \sum_{k=1}^{n-1} k! \stirlingii{n-1}{k} && \text{(Let $j=k-1$; use $\stirlingii{n-1}{n}=0$)}\\
&= \operatorname{Bell}(n-1) + \sum_{k=0}^{n-1} k! \stirlingii{n-1}{k} && \text{(Since $\stirlingii{n-1}{0}=0$ for $n\ge2$)}\\
&= \operatorname{Bell}(n-1) + \operatorname{Bell}(n-1) = 2\operatorname{Bell}(n-1).
\end{align*}
\end{proof}

The asymptotic size is determined by the EGF for $(\operatorname{Bell}(m))$.

\begin{lemma}
\label{lemma:egf}
The exponential generating function of $(\operatorname{Bell}(m))$ is $A(x) \coloneqq \sum_{m=0}^{\infty} \operatorname{Bell}(m) \frac{x^m}{m!} = \frac{1}{2 - e^x}$.
\end{lemma}
\begin{proof}
Using the standard EGF for Stirling numbers, $\sum_{m \ge k} \stirlingii{m}{k} \frac{x^m}{m!} = \frac{(e^x-1)^k}{k!}$, and interchanging summation:
\begin{align*}
A(x) = \sum_{k=0}^{\infty} k! \left( \sum_{m=k}^{\infty} \stirlingii{m}{k} \frac{x^m}{m!} \right)
= \sum_{k=0}^{\infty} k! \left( \frac{(e^x-1)^k}{k!} \right) = \sum_{k=0}^{\infty} (e^x-1)^k
= \frac{1}{1-(e^x-1)} = \frac{1}{2-e^x}.
\end{align*}
\end{proof}

\begin{proposition}
\label{prop:bell_asymptotics}
The ordered Bell numbers have the asymptotic behavior $\operatorname{Bell}(m) \sim \frac{m!}{2(\ln 2)^{m+1}}$ as $m \to \infty$.
\end{proposition}
\begin{proof}
The EGF $A(x) = (2-e^x)^{-1}$ has its dominant singularity at $x=\rho=\ln 2$. This is a simple pole with residue $\operatorname{Res}(A; \ln 2) = -1/2$. Applying a standard transfer theorem from analytic combinatorics yields the result. See, e.g., Flajolet and Sedgewick, \emph{Analytic Combinatorics} (2009), Theorem VI.1. \cite{flajolet2009analytic}.
\end{proof}

\subsection{The Centered Case: Partitions without Singletons}
\label{subsec:centered-analysis}

The coefficient $\Ccen{n}$ from \Cref{thm:central_sharp} arises from the subset of partitions where every block has size at least $2$. Let $\PP_{\ge 2}(n)$ denote the set of such partitions of $[n]$. The coefficient is defined as
$\Ccen{n} \coloneqq \sum_{\pi \in \PP_{\ge 2}(n)} (\abs{\pi}-1)!.$

\begin{lemma}
\label{lemma:egf_C0}
The exponential generating function for the sequence $(\Ccen{n})_{n\ge 1}$ is
$A_{\mathrm{cen}}(x) \coloneqq \sum_{n\ge 1} \Ccen{n} \frac{x^n}{n!} = -\log\bigl(2 - e^x + x\bigr).$
\end{lemma}
\begin{proof}
Let
$B_{\ge 2}(x) \coloneqq \sum_{k=2}^{\infty} \frac{x^k}{k!} = e^x - 1 - x$
be the EGF of a single allowed block. A partition into exactly $k$ allowed blocks is an unordered set of $k$ such blocks, so its EGF is $(B_{\ge2}(x))^k/k!$. Grouping partitions by their number of blocks therefore gives
\[
\Ccen{n}=\sum_{k\ge1} (k-1)!\,[x^n/n!]\frac{(B_{\ge2}(x))^k}{k!}.
\]
Summing over $n$ reconstructs the exponential generating function, hence
\[
A_{\mathrm{cen}}(x)=\sum_{k\ge1}(k-1)!\frac{\bigl(B_{\ge2}(x)\bigr)^k}{k!}
=\sum_{k\ge1}\frac{\bigl(B_{\ge2}(x)\bigr)^k}{k}
=-\log\bigl(1-B_{\ge2}(x)\bigr),
\]
which simplifies to the stated form.
\end{proof}

\begin{lemma}
\label{lem:centered-dominant}
Let
$B_{\ge 2}(z)=e^z-1-z=\sum_{m=2}^{\infty}\frac{z^m}{m!},$
and let $\rho_{\mathrm{cen}}>0$ be the unique positive solution of
$B_{\ge 2}(\rho_{\mathrm{cen}})=1,$
equivalently $2-e^{\rho_{\mathrm{cen}}}+\rho_{\mathrm{cen}}=0$.
Then $1-B_{\ge 2}(z)\neq 0$ for all $\abs{z}<\rho_{\mathrm{cen}}$, and on the circle
$\abs{z}=\rho_{\mathrm{cen}}$ the equation $B_{\ge 2}(z)=1$ holds only at $z=\rho_{\mathrm{cen}}$.
Consequently, $\rho_{\mathrm{cen}}$ is the unique dominant singularity of
$A_{\mathrm{cen}}(z)=-\log\!\bigl(1-B_{\ge 2}(z)\bigr).$
\end{lemma}

\begin{proof}
Since $B_{\ge 2}$ has nonnegative Taylor coefficients and $B_{\ge 2}(0)=0$, the function
$r\mapsto B_{\ge 2}(r)$ is strictly increasing for $r>0$. Hence for $\abs{z}<\rho_{\mathrm{cen}}$,
\[
\abs{B_{\ge 2}(z)}
\le
B_{\ge 2}(\abs{z})
<
B_{\ge 2}(\rho_{\mathrm{cen}})
=1,
\]
so $1-B_{\ge 2}(z)\neq 0$ in the open disk.

Now let $\abs{z}=\rho_{\mathrm{cen}}$ and suppose that $B_{\ge 2}(z)=1$. Then
\[
1=\abs{B_{\ge 2}(z)}
\le
\sum_{m=2}^{\infty}\frac{\abs{z}^m}{m!}
=
B_{\ge 2}(\rho_{\mathrm{cen}})
=1,
\]
so equality holds in the triangle inequality. Since the coefficients of $z^2$ and $z^3$
in $B_{\ge 2}(z)$ are both strictly positive, equality forces $z^2$ and $z^3$ to have the same argument.
Therefore $z/\abs{z}=1$, and hence $z=\rho_{\mathrm{cen}}$.
\end{proof}

\begin{proposition}
\label{prop:C0-asymptotics}

Let $\rho_{\mathrm{cen}}$ be the unique positive solution to $2 - e^{\rho_{\mathrm{cen}}} + \rho_{\mathrm{cen}} = 0$ ($\rho_{\mathrm{cen}} \approx 1.146$). Then
$\Ccen{n} \sim \frac{(n-1)!}{\rho_{\mathrm{cen}}^n} \qquad (n\to\infty).$
\end{proposition}
\begin{proof}
By \Cref{lem:centered-dominant}, the unique dominant singularity of
$A_{\mathrm{cen}}(x)=-\log\bigl(2-e^x+x\bigr)$
is the positive real number $\rho_{\mathrm{cen}}$ satisfying
$2-e^{\rho_{\mathrm{cen}}}+\rho_{\mathrm{cen}}=0.$
Since
\[
\frac{d}{dx}\bigl(2-e^x+x\bigr)\Big|_{x=\rho_{\mathrm{cen}}}
=
1-e^{\rho_{\mathrm{cen}}}
\neq 0,
\]
the zero at $x=\rho_{\mathrm{cen}}$ is simple. Therefore, as $x\to \rho_{\mathrm{cen}}$,
\[
2-e^x+x
=
(1-e^{\rho_{\mathrm{cen}}})(x-\rho_{\mathrm{cen}})
+O\bigl((x-\rho_{\mathrm{cen}})^2\bigr)
=
C(\rho_{\mathrm{cen}}-x)\bigl(1+o(1)\bigr),
\]
where $C=e^{\rho_{\mathrm{cen}}}-1>0$. Hence
\[
A_{\mathrm{cen}}(x)
=
-\log\bigl(2-e^x+x\bigr)
=
-\log\!\left(1-\frac{x}{\rho_{\mathrm{cen}}}\right)
+O(1)
\qquad (x\to \rho_{\mathrm{cen}}).
\]
Applying standard singularity analysis for logarithmic singularities
(see \cite[Theorem VI.2]{flajolet2009analytic}) yields
$[x^n]A_{\mathrm{cen}}(x)\sim \frac{1}{n\,\rho_{\mathrm{cen}}^n}.$
Since
$A_{\mathrm{cen}}(x)=\sum_{n\ge 1}\Ccen{n}\frac{x^n}{n!},$
it follows that
$\Ccen{n}\sim \frac{(n-1)!}{\rho_{\mathrm{cen}}^n}.$
\end{proof}

\subsection{The Symmetric Case: Partitions with Even-Sized Blocks}
\label{subsec:Csym-analysis}

The coefficient $\Csym{n}$ from \Cref{thm:symmetric-refinement} arises from the subset of partitions where every block has even cardinality. Let $\PP_{\mathrm{even}}(n)$ denote the set of such partitions of $[n]$. The coefficient is defined as
$\Csym{n} \coloneqq \sum_{\pi \in \PP_{\mathrm{even}}(n)} (\abs{\pi}-1)!.$

\begin{lemma}
\label{lemma:egf_Csym}
The exponential generating function for the sequence $(\Csym{n})_{n\ge 1}$ is
$A_{\mathrm{sym}}(x) \coloneqq \sum_{n\ge 1} \Csym{n} \frac{x^n}{n!} = -\log\bigl(2-\cosh x\bigr).$
\end{lemma}
\begin{proof}
Let
$B_{\mathrm{even}}(x) \coloneqq \sum_{m\ge 1} \frac{x^{2m}}{(2m)!} = \cosh x - 1$
be the EGF of a single allowed block. A partition into exactly $k$ even-sized blocks is an unordered set of $k$ such blocks, so its EGF is $(B_{\mathrm{even}}(x))^k/k!$. Grouping partitions by their number of blocks therefore gives
\[
\Csym{n}=\sum_{k\ge1} (k-1)!\,[x^n/n!]\frac{(B_{\mathrm{even}}(x))^k}{k!}.
\]
Summing over $n$ reconstructs the exponential generating function, hence
\[
A_{\mathrm{sym}}(x)=\sum_{k\ge1}(k-1)!\frac{\bigl(B_{\mathrm{even}}(x)\bigr)^k}{k!}
=\sum_{k\ge1}\frac{\bigl(B_{\mathrm{even}}(x)\bigr)^k}{k}
=-\log\bigl(1-B_{\mathrm{even}}(x)\bigr),
\]
which simplifies to the stated form.
\end{proof}

\begin{lemma}
\label{lem:symmetric-dominant}
Let
$B_{\mathrm{even}}(z)=\cosh z-1=\sum_{m=1}^{\infty}\frac{z^{2m}}{(2m)!},$
and let $\rho_{\mathrm{sym}}=\operatorname{arcosh}2$, so that
$B_{\mathrm{even}}(\rho_{\mathrm{sym}})=1.$
Then $1-B_{\mathrm{even}}(z)\neq 0$ for all $\abs{z}<\rho_{\mathrm{sym}}$, and on the circle
$\abs{z}=\rho_{\mathrm{sym}}$ the equation $B_{\mathrm{even}}(z)=1$ holds only at
$z=\pm \rho_{\mathrm{sym}}$. Consequently, the dominant singularities of
$A_{\mathrm{sym}}(z)=-\log\!\bigl(1-B_{\mathrm{even}}(z)\bigr)$
are exactly $\pm \rho_{\mathrm{sym}}$.
\end{lemma}

\begin{proof}
Since $B_{\mathrm{even}}$ has nonnegative Taylor coefficients and $B_{\mathrm{even}}(0)=0$,
for $\abs{z}<\rho_{\mathrm{sym}}$ we have
\[
\abs{B_{\mathrm{even}}(z)}
\le
B_{\mathrm{even}}(\abs{z})
<
B_{\mathrm{even}}(\rho_{\mathrm{sym}})
=1,
\]
so $1-B_{\mathrm{even}}(z)\neq 0$ in the open disk.

Now let $\abs{z}=\rho_{\mathrm{sym}}$ and suppose that $B_{\mathrm{even}}(z)=1$. Then
\[
1=\abs{B_{\mathrm{even}}(z)}
\le
\sum_{m=1}^{\infty}\frac{\abs{z}^{2m}}{(2m)!}
=
B_{\mathrm{even}}(\rho_{\mathrm{sym}})
=1,
\]
so equality holds in the triangle inequality. Since the coefficients of $z^2$ and $z^4$
are both strictly positive, equality forces $z^2$ and $z^4$ to have the same argument.
Thus $z^2/\abs{z}^2=1$, which implies $z=\pm \rho_{\mathrm{sym}}$.
\end{proof}

\begin{proposition}
\label{prop:Csym-asymptotics}
Let $\rho_{\mathrm{sym}} = \operatorname{arcosh} 2 = \ln(2+\sqrt{3})$ ($\rho_{\mathrm{sym}} \approx 1.317$). Then, for even $n=2m$,
$\Csym{2m} \sim 2\frac{(2m-1)!}{\rho_{\mathrm{sym}}^{2m}} \qquad (m\to\infty),$
and $\Csym{n}=0$ for all odd $n$.
\end{proposition}
\begin{proof}
By \Cref{lem:symmetric-dominant}, the dominant singularities of
$A_{\mathrm{sym}}(x)=-\log\bigl(2-\cosh x\bigr)$
are exactly $x=\pm \rho_{\mathrm{sym}}$, where $\rho_{\mathrm{sym}}=\operatorname{arcosh}2$.

Since
\[
\frac{d}{dx}\bigl(2-\cosh x\bigr)\Big|_{x=\rho_{\mathrm{sym}}}
=
-\sinh(\rho_{\mathrm{sym}})
\neq 0,
\]
the zero at $x=\rho_{\mathrm{sym}}$ is simple. Thus, as $x\to \rho_{\mathrm{sym}}$,
\[
2-\cosh x
=
-\sinh(\rho_{\mathrm{sym}})(x-\rho_{\mathrm{sym}})
+O\bigl((x-\rho_{\mathrm{sym}})^2\bigr)
=
C(\rho_{\mathrm{sym}}-x)\bigl(1+o(1)\bigr),
\]
where $C=\sinh(\rho_{\mathrm{sym}})>0$. Hence
\[
A_{\mathrm{sym}}(x)
=
-\log\bigl(2-\cosh x\bigr)
=
-\log\!\left(1-\frac{x}{\rho_{\mathrm{sym}}}\right)
+O(1)
\qquad (x\to \rho_{\mathrm{sym}}).
\]
Similarly, as $x\to -\rho_{\mathrm{sym}}$,
$A_{\mathrm{sym}}(x)
=
-\log\!\left(1+\frac{x}{\rho_{\mathrm{sym}}}\right)
+O(1).$
Therefore, by standard singularity analysis for logarithmic singularities,
\[
[x^{2m}]A_{\mathrm{sym}}(x)
\sim
\frac{1}{2m\,\rho_{\mathrm{sym}}^{2m}}
+
\frac{1}{2m\,(-\rho_{\mathrm{sym}})^{2m}}
=
\frac{2}{2m\,\rho_{\mathrm{sym}}^{2m}}.
\]
Since
$A_{\mathrm{sym}}(x)=\sum_{n\ge 1}\Csym{n}\frac{x^n}{n!},$
it follows that, for even $n=2m$,
$\Csym{2m}\sim 2\frac{(2m-1)!}{\rho_{\mathrm{sym}}^{2m}}.$
Finally, $\Csym{n}=0$ for odd $n$ because no partition of an odd integer can consist solely of even-sized blocks.
\end{proof}

\section{Impossibility of a Converse Bound}
\label{sec:converse}

A natural question is whether a universal reverse inequality exists, that is, whether $\abs{\kappa_n(X)} \ge D_n \E{\abs{X}^n}$ for some constant $D_n > 0$. The answer is negative for all $n \ge 2$.

\begin{remark}
\label{rem:no-converse}
For any integer $n \ge 2$, there exists no constant $D_n > 0$ such that
$\abs{\kappa_n(X)} \ge D_n \E{\abs{X}^n}$
holds for all real-valued random variables $X$ with finite $n$-th moment.
\end{remark}
\begin{proof}
For $n=2$, take the constant random variable $X\equiv 1$. Then
$\kappa_2(X)=0,
\qquad
\E{\abs{X}^2}=1,$
so no such constant $D_2>0$ can exist.

Now fix $n\ge 3$ and let $X\sim \mathcal{N}(0,1)$. Then $\kappa_n(X)=0$, while
$\E{\abs{X}^n}>0.$
Thus no such constant $D_n>0$ can exist for $n\ge 3$ either.
\end{proof}

This impossibility arises because algebraic cancellations in the moment--cumulant formula can force $\kappa_n$ to vanish even when moments are large. However, a converse-type bound is possible for the \emph{signed} raw or central moment if we control that moment using the maximum of the normalized cumulants up to order $n$.

Let $\PP_0(n)$ denote the set of partitions of $[n]$ with no singleton blocks.
Let $B_n = |\PP(n)|$ be the ordinary Bell number, and let $B_n^{(0)} = |\PP_0(n)|$ be the number of partitions with no singleton blocks. Define the \emph{cumulant envelope} as $\mathcal{K}_n(X) \coloneqq \max_{1 \le j \le n} \abs{\kappa_j(X)}^{n/j}$.

\begin{proposition}
\label{prop:raw-lower}
For any $n \ge 1$ and $X$ with $\E{\abs{X}^n} < \infty$, the raw moment $m_n = \E{X^n}$ satisfies
\begin{equation}
\label{eq:raw-lower}
\abs{m_n} \le B_n \mathcal{K}_n(X).
\end{equation}
\end{proposition}
\begin{proof}
The inverse moment--cumulant formula is
\[
m_n = \sum_{\pi \in \PP(n)} \prod_{B \in \pi} \kappa_{\abs{B}}(X)
\]
(see, e.g., \cite{leonov1959method, speed1983cumulants, mccullagh1987tensor}). By definition,
$\abs{\kappa_j(X)} \le \mathcal{K}_n(X)^{j/n}$. Thus, for any partition $\pi$, the product term is bounded by
$\prod_{B \in \pi} \mathcal{K}_n(X)^{\abs{B}/n} = \mathcal{K}_n(X)^{\sum \abs{B}/n} = \mathcal{K}_n(X)$.
Summing over the $B_n$ partitions in $\PP(n)$ yields the result.
\end{proof}

\begin{proposition}
\label{prop:centered-lower}
For any $n \ge 2$, the central moment $m_n^{(c)} = \E{(X-\E{X})^n}$ satisfies
\begin{equation}
\label{eq:centered-lower}
\abs{m_n^{(c)}} \le B_n^{(0)} \mathcal{K}_n(X).
\end{equation}
\end{proposition}
\begin{proof}
Applying the inverse formula to the centered variable $Y=X-\E{X}$, we note that $\kappa_1(Y)=0$. Consequently, any partition containing a singleton vanishes. The sum restricts to $\PP_0(n)$, and the bound follows by the same argument as in \Cref{prop:raw-lower}.
\end{proof}

\begin{remark}
\label{rem:gaussian-salvaged}
These bounds resolve the contradiction in \Cref{rem:no-converse}. For a Gaussian $X \sim \mathcal{N}(0, \sigma^2)$ and $n=4$, $\kappa_4(X)=0$. A bound depending solely on $\kappa_4$ fails. However, \Cref{prop:centered-lower} relies on the envelope: $\mathcal{K}_4(X) = \max(0, \abs{\kappa_2}^2, 0, 0) = \sigma^4$. We correctly obtain $m_4^{(c)} = 3\sigma^4 \le B_4^{(0)} \sigma^4$, which holds since $B_4^{(0)}=4$ and $3 \le 4$.
\end{remark}

\section{Universal Bounds for Joint Cumulants}
\label{sec:multivariate}

The bounds extend naturally to joint cumulants.

\subsection{Setup and Partition Formula}

Let $\bm X=(X_1,\dots,X_d)$ be a random vector and let $\bm\nu=(\nu_1,\dots,\nu_d)\in\Nzero^d$ be a multi-index of total order $N=\sum_{j=1}^d \nu_j$. The joint cumulant $\kappa_{\bm\nu}(\bm X)$ is defined via the partition formula on a set of $N$ indices corresponding to the variables. To formalize this, consider a mapping $\sigma: \{1, \dots, N\} \to \{1, \dots, d\}$ that associates each integer with a variable index, such that $|\sigma^{-1}(j)| = \nu_j$ for each $j \in \{1,\dots,d\}$. The joint cumulant admits a partition representation analogous to \Cref{lem:mom-cum} (see, e.g., \cite{mccullagh1987tensor}):
\begin{equation}
\label{eq:multi-partition}
\kappa_{\bm\nu}(\bm X) = \sum_{\pi\in \PP(N)} (-1)^{\abs{\pi}-1}(\abs{\pi}-1)! \prod_{B\in\pi} \mathsf{E}\Biggl[\prod_{k \in B} X_{\sigma(k)}\Biggr].
\end{equation}
The inner expectation can be written more compactly as $\E{\prod_{j=1}^d X_j^{n_{j,B}}}$, where $n_{j,B} = |\{k \in B : \sigma(k)=j\}|$ is the number of indices in block $B$ that correspond to variable $X_j$.

\begin{remark}
\label{rem:multi-shift}
For total order $N\ge 2$, $\kappa_{\bm\nu}(X_1,\dots,X_d)=\kappa_{\bm\nu}(X_1-\E{X_1},\dots,X_d-\E{X_d})$. Thus, all centered multivariate refinements below apply to any $\bm X$ after replacing each component by its centered version.
\end{remark}

\subsection{Blockwise H\"older and Product Collapse}

\begin{lemma}
\label{lem:multi-holder}
Assume $\E{\abs{X_j}^N}<\infty$ for all $j$. For any block $B$,
\[
\Biggl|\mathsf{E}\Bigl[\prod_{j=1}^d X_j^{n_{j,B}}\Bigr]\Biggr|
\le \prod_{j: n_{j,B}>0} \bigl(\E{\abs{X_j}^N}\bigr)^{n_{j,B}/N}.
\]
\end{lemma}

\begin{proof}
Fix a block $B$. Let $J_B=\{j: n_{j,B}>0\}$ and $|B|=\sum_{j\in J_B} n_{j,B}$. For $j\in J_B$ set $p_{j,B}=N/n_{j,B}$. Then $\sum_{j\in J_B} 1/p_{j,B} = |B|/N \le 1$. If $|B|=N$, the sum equals $1$ and generalized H\"older's inequality gives
\[
\mathsf{E}\Biggl[\prod_{j\in J_B} |X_j|^{n_{j,B}}\Biggr] \le \prod_{j\in J_B} \norm*{|X_j|^{n_{j,B}}}_{p_{j,B}} = \prod_{j\in J_B} \bigl(\E{|X_j|^N}\bigr)^{n_{j,B}/N}.
\]
If $|B|<N$, we apply generalized H\"older to the family $\{|X_j|^{n_{j,B}}\}_{j\in J_B}$ together with the constant function $1$, using the exponents $\{p_{j,B}\}_{j\in J_B}$ and $r_B = N/(N-|B|)$, which satisfy $\sum_{j\in J_B} (1/p_{j,B}) + 1/r_B = 1$. Since $\norm{1}_{r_B}=1$ on a probability space, the same bound follows. Taking absolute values completes the proof.
\end{proof}

\begin{lemma}
\label{lem:multi-collapse}
Under the assumptions of \Cref{lem:multi-holder},
\[
\prod_{B\in\pi} \Biggl|\mathsf{E}\Bigl[\prod_{j=1}^d X_j^{n_{j,B}}\Bigr]\Biggr|
\le \prod_{j=1}^d \bigl(\E{\abs{X_j}^N}\bigr)^{\nu_j/N}.
\]
\end{lemma}
\begin{proof}
Apply \Cref{lem:multi-holder} to each block $B$ and multiply over $B\in\pi$. The exponent for $\E{|X_j|^N}$ accumulates as $\sum_{B\in\pi} n_{j,B}/N = \nu_j/N$, yielding the result.
\end{proof}

\subsection{Unified Multivariate Bounds}

The following theorem generalizes \Cref{thm:main}, \Cref{thm:central_sharp}, and \Cref{thm:symmetric-refinement} to the multivariate setting.

\begin{theorem}
\label{thm:multi-unified}
Let $\bm X=(X_1,\dots,X_d)$ be a random vector and $\bm\nu$ a multi-index with total order $N \ge 1$. Assume finite $N$-th absolute moments for all components. The joint cumulant satisfies the bound
\begin{equation}
\label{eq:multi-general-bound}
\abs{\kappa_{\bm\nu}(\bm X)} \le C_N \prod_{j=1}^d \left( M_N(X_j) \right)^{\nu_j/N},
\end{equation}
where the coefficient $C_N$ and the moment functional $M_N$ are determined by the structural properties of $\bm X$:
\begin{enumerate}[label=(\roman*), leftmargin=*]
    \item \textbf{Raw Moments:} For any $\bm X$, we have $C_N = \Craw{N}$ (as in \Cref{thm:main}) and $M_N(X) = \E{\abs{X}^N}$.
    \item \textbf{Central Moments:} If $N \ge 2$, we have $C_N = \Ccen{N}$ (as in \Cref{thm:central_sharp}) and $M_N(X) = \E{\abs{X-\E{X}}^N}$.
    \item \textbf{Symmetric Distributions:} If $\bm X \overset{d}= -\bm X$, then $\kappa_{\bm\nu}(\bm X)=0$ for odd $N$. For even $N$, we have $C_N = \Csym{N}$ (as in \Cref{thm:symmetric-refinement}) and $M_N(X) = \E{\abs{X}^N}$.
\end{enumerate}
\end{theorem}

\begin{proof}
The proof follows the univariate argument mutatis mutandis. Taking the absolute value of \eqref{eq:multi-partition} yields
\[
\abs{\kappa_{\bm\nu}(\bm X)} \le \sum_{\pi \in \mathcal{S}} (\abs{\pi}-1)! \prod_{B\in\pi} \Biggl|\mathsf{E}\Bigl[\prod_{k \in B} \tilde{X}_{\sigma(k)}\Bigr]\Biggr|,
\]
where $\tilde{X}$ represents either the raw or centered variable, and $\mathcal{S} \subseteq \PP(N)$ is the subset of valid partitions. Applying \Cref{lem:multi-collapse} allows the product of moments to be factored out uniformly:
\[
\abs{\kappa_{\bm\nu}(\bm X)} \le \left( \sum_{\pi \in \mathcal{S}} (\abs{\pi}-1)! \right) \prod_{j=1}^d \bigl(\E{\abs{\tilde{X}_j}^N}\bigr)^{\nu_j/N}.
\]
The summation term corresponds exactly to the univariate coefficients $\Craw{N}$, $\Ccen{N}$, or $\Csym{N}$ derived in \Cref{sec:coefficient_analysis}, depending on whether $\mathcal{S}$ is the set of all partitions, partitions without singletons (due to centering), or partitions with even blocks (due to symmetry). In the symmetric case, note that for any block $B$ the monomial $\prod_{k\in B} X_{\sigma(k)}$ has total degree $|B|$; hence if $|B|$ is odd and $\bm X\overset d= -\bm X$, its expectation vanishes.
\end{proof}

\section{Non-attainment and Comparison with a Unit-Circle Benchmark}
\label{sec:strict}

In this section, we address the sharpness of the proposed bounds. We first demonstrate strictness whenever the relevant coefficient is positive, with the centered variance case $n=2$ as the unique identity case. We then compare our sharpest symmetric rate with the benchmark provided by the Rademacher law, which belongs to the extremal class $|X|=1$ a.s.

\subsection{Strictness of the Upper Bounds}

A natural question is whether there exists a worst-case distribution for which the bounds in \Cref{thm:main}, \Cref{thm:central_sharp}, or \Cref{thm:symmetric-refinement} become equalities. We show that, except in the centered variance identity and the degenerate symmetric odd-order zero-coefficient case, equality is impossible for non-degenerate laws.

First, we note the exceptional identity: for the centered refinement with $n=2$,
$\kappa_2(X) = \Var{X} = \E{\abs{X-\E{X}}^2},$
and since $\Ccen{2}=1$, the inequality $\abs{\kappa_2(X)} \le \Ccen{2}\,\E{\abs{X-\E{X}}^2}$ is an identity. In the symmetric odd-order case the coefficient is zero and the cumulant vanishes identically; outside these exceptional situations, the bounds are strict whenever the relevant coefficient is positive.

\begin{proposition}
\label{prop:strictness}
Let $X$ be a non-degenerate random variable (i.e., $X$ is not almost surely constant).
\begin{enumerate}[label=(\roman*), leftmargin=*, itemsep=2pt]
  \item For the raw-moment bound (\Cref{thm:main}), if $n\ge 2$, then $\abs{\kappa_n(X)} < \Craw{n}\,\E{\abs{X}^n}$.
  \item For the universal central bound (\Cref{thm:central_sharp}), if $n\ge 3$, then $\abs{\kappa_n(X)} < \Ccen{n}\,\E{\abs{X-\E{X}}^n}$.
  \item For the symmetric refinement (\Cref{thm:symmetric-refinement}), if $X$ is symmetric and $n\ge 4$ is even, then $\abs{\kappa_n(X)} < \Csym{n}\,\E{\abs{X}^{n}}$.
\end{enumerate}
\end{proposition}

\begin{proof}
We treat the three bounds separately.

\smallskip
\noindent
\emph{(i) Raw bound.}
Let
\[
S_n(X):=\sum_{\pi\in\PP(n)}(-1)^{|\pi|-1}(|\pi|-1)!\prod_{B\in\pi}m_{|B|},
\]
so that $S_n(X)=\kappa_n(X)$. The proof of \Cref{thm:main} bounds $\abs{S_n(X)}$ by applying:
\begin{enumerate}[label=(\alph*), leftmargin=*]
    \item the triangle inequality;
    \item the estimate $\abs{m_k}\le \mu_k$;
    \item the product collapse \Cref{lem:product_collapse}.
\end{enumerate}

Suppose first that $\abs{X}$ is not almost surely constant. Then equality in Lyapunov's inequality
fails for every pair of distinct orders by \Cref{lem:lyapunov-equality}, hence for every $1\le j<n$,
$\mu_j<\mu_n^{j/n}.$
Therefore, for every partition $\pi$ with at least two blocks,
$\prod_{B\in\pi}\mu_{|B|}<\mu_n.$
Since such partitions occur for every $n\ge2$, step (c) is strict for at least one term appearing in
the coefficient mass, so equality in the final bound is impossible.

It remains to consider the case $\abs{X}=c>0$ almost surely. Since $X$ is real-valued and non-degenerate,
necessarily $X\in\{-c,c\}$ almost surely.

If $n=2$, then
$\abs{\kappa_2(X)}=\Var{X}<c^2<2c^2=\Craw{2}\,\E{\abs{X}^2},$
because $\E{\abs{X}^2}=c^2$ and non-degeneracy implies $\Var{X}<c^2$.

Now assume $n\ge3$. If $n$ is odd and $X$ is symmetric, then $\kappa_n(X)=0$, so the inequality is strict.
If $n$ is odd and $X$ is not symmetric, then $m_1\neq0$ and $m_{n-1}=c^{n-1}>0$. In the partition formula,
the one-block partition contributes the nonzero term $m_n$, while every partition of type $(1,n-1)$
contributes the nonzero term $-m_1m_{n-1}$. These two terms have opposite signs, so equality in the
triangle inequality is impossible.

If $n$ is even, then $m_2=c^2>0$ and $m_{n-2}=c^{n-2}>0$. The one-block partition contributes $m_n$,
while every partition of type $(2,n-2)$ contributes $-m_2m_{n-2}$. Again the signs are opposite, so
equality in the triangle inequality is impossible. This proves strictness in the raw case.

\smallskip
\noindent
\emph{(ii) Centered bound.}
Let $Y:=X-\E{X}$. Then $\kappa_n(X)=\kappa_n(Y)$ for all $n\ge2$, and \Cref{thm:central_sharp} is obtained by
applying the same three steps as above to the restricted partition sum over partitions with no singletons.

If $n=3$, then the only partition of $[3]$ with no singleton blocks is the one-block partition, so
$\kappa_3(Y)=m_3(Y).$
Hence
$\abs{\kappa_3(Y)}=\abs{\E{Y^3}}\le \E{\abs{Y}^3}.$
Equality in $\abs{\E{Y^3}}=\E{\abs{Y}^3}$ would force $Y^3$ to have constant sign almost surely, hence
$Y$ to be almost surely nonnegative or almost surely nonpositive. Since $\E{Y}=0$ and $Y$ is non-degenerate,
this is impossible. Therefore the inequality is strict for $n=3$.

Now let $n\ge4$. If $\abs{Y}$ is not almost surely constant, then as above
$\prod_{B\in\pi}\E{\abs{Y}^{|B|}}<\E{\abs{Y}^n}$
for every partition $\pi$ with at least two blocks. Since there exist partitions of $[n]$ with no singleton
blocks and at least two blocks for every $n\ge4$ (for instance, a partition of type $(2,n-2)$), step (c) is
strict for at least one admissible term, so equality in the final bound is impossible.

If instead $\abs{Y}=c>0$ almost surely, then since $\E{Y}=0$ and $Y$ is real-valued, necessarily
$\Pr{Y=c}=\Pr{Y=-c}=\frac12,$
so $Y$ is symmetric. For odd $n$, we have $\kappa_n(Y)=0$, hence strictness is immediate. For even $n\ge4$,
the one-block partition contributes $m_n(Y)=c^n$, while every partition of type $(2,n-2)$ contributes
$-m_2(Y)m_{n-2}(Y)=-c^n$. These nonzero terms have opposite signs, so equality in the triangle inequality
is impossible. This proves strictness in the centered case for all $n\ge3$.

\smallskip
\noindent
\emph{(iii) Symmetric bound.}
Assume now that $X$ is symmetric and $n\ge4$ is even. If $\abs{X}$ is not almost surely constant, then
the same strict Lyapunov argument shows that step (c) is strict for every admissible partition with at least
two blocks, and such partitions exist for every even $n\ge4$ (for instance, of type $(2,n-2)$).

If instead $\abs{X}=c>0$ almost surely, then symmetry implies that $X$ is the Rademacher law on $\{\pm c\}$.
The one-block partition contributes $m_n=c^n$, while every partition of type $(2,n-2)$ contributes
$-m_2m_{n-2}=-c^n$. These nonzero terms have opposite signs, so equality in the triangle inequality
is impossible. Thus the symmetric bound is strict as well.
\end{proof}

\subsection{Comparison with a unit-circle benchmark}

While the bounds are strict, it is natural to compare their exponential rates with explicit distributions for which the cumulants can be analyzed exactly. The cleanest benchmark comes from laws with unit absolute moments.

\begin{definition}
Let $\mathcal{M}_1$ be the class of scalar random variables such that $\E{\abs{X}^n} = 1$ for all integers $n \ge 1$.
\end{definition}

\begin{lemma}
If $X \in \mathcal{M}_1$, then $\abs{X} = 1$ almost surely.
\end{lemma}
\begin{proof}
Consider the variance of the random variable $Y = \abs{X}$. By the definition of $\mathcal{M}_1$, we have $\E{Y} = \E{\abs{X}} = 1$ and $\E{Y^2} = \E{\abs{X}^2} = 1$. Thus,
$\Var{Y} = \E{Y^2} - (\E{Y})^2 = 1 - 1^2 = 0.$
A random variable with zero variance is almost surely constant. Hence $Y = 1$ a.s., which implies $\abs{X} = 1$ a.s.
\end{proof}

For this class of variables, we can determine the exact growth rate of the cumulants in a concrete example.

\begin{example}
Let $X$ be a symmetric random variable in $\mathcal{M}_1$. The distribution of $X$ is uniquely determined as the Rademacher law: $\Pr{X=1}=\Pr{X=-1}=1/2$. The cumulant generating function is
$K_X(t) = \log \E{e^{tX}} = \log\left(\frac{e^t + e^{-t}}{2}\right) = \log(\cosh t).$
Since $X$ is symmetric, $\kappa_n(X) = 0$ for odd $n$. For even $n$, the cumulants are related to the tangent numbers (or Bernoulli numbers). However, their asymptotic growth is most easily derived via the singularities of $K_X(t)$.

The function $\log(\cosh t)$ is analytic everywhere in the complex plane except where $\cosh t = 0$. The zeros of $\cosh t$ are located at $t_k = i \frac{\pi}{2} (2k+1)$ for $k \in \mathbb{Z}$. The singularities closest to the origin are the logarithmic singularities of $K_X$ (equivalently, the simple poles of $K_X'(t)=\tanh t$) at $t_0 = i\pi/2$ and $\bar{t}_0 = -i\pi/2$. The radius of convergence of the Taylor series is therefore exactly $R = \pi/2$.

Near the singularity $t_0 = i\pi/2$, we have the behavior
$\cosh(t) = \cosh(i\pi/2 + (t-t_0)) = i \sinh(t-t_0) \approx i(t-t_0).$
Thus, the singular part of the generating function behaves as
$K_X(t) \sim \log(t - t_0) = \log\left(1 - \frac{t}{t_0}\right) + C.$
Applying standard singularity analysis as above, the coefficient
$[t^n]K_X(t)=\frac{\kappa_n}{n!}$
satisfies
\[
  \frac{\kappa_n}{n!}
  \sim
  -\frac{1}{n}\left(t_0^{-n}+\bar t_0^{-n}\right).
\]
For even $n=2m$, the contributions from the conjugate singularities add to
\[
  \frac{\kappa_{2m}}{(2m)!}
  \sim
  -\frac{2}{2m}\left(\frac{2}{\pi}\right)^{2m}(-1)^m
  =
  (-1)^{m+1}\frac{2}{2m}\left(\frac{2}{\pi}\right)^{2m}.
\]
Taking absolute values and rearranging yields
\begin{equation}
\label{eq:rademacher_rate}
  \abs{\kappa_{2m}(X)} \sim 2 (2m-1)! \left( \frac{1}{\pi/2} \right)^{2m}.
\end{equation}
\end{example}

We can now compare this exact rate with our symmetric bound from \Cref{thm:symmetric-refinement}. Our bound states that $\abs{\kappa_{2m}} \le \Csym{2m} \E{\abs{X}^{2m}}$, where $\E{\abs{X}^{2m}}=1$. The coefficient behaves as
\[
  \Csym{2m} \sim 2 (2m-1)! \left( \frac{1}{\rho_{\mathrm{sym}}} \right)^{2m}, \quad \text{with } \rho_{\mathrm{sym}} = \operatorname{arcosh} 2 \approx 1.317.
\]
Comparing this to \eqref{eq:rademacher_rate}, we see that the structural forms are identical (both scale as factorials), but the exponential rates differ. The ratio of the base parameters is
\[
  \frac{\rho_{\mathrm{sym}}}{\pi/2} \approx \frac{1.317}{1.571} \approx 0.838.
\]
Thus the sharpest universal symmetric constant in this paper captures the Rademacher benchmark up to a factor of about $0.84$ in the exponential rate.

The following theorem gives a general zero-free disk for probability measures supported on the unit circle. It does not identify the optimal universal radius over that class; we use it only as a benchmark result.

In the following theorem we temporarily allow the random variable to be complex-valued, supported on the unit circle
$\mathbb{T}=\{z\in\mathbb{C} : \abs{z}=1\}.$

\begin{theorem}
\label{thm:zero_free_disk}
Let $\mu$ be any probability measure on the unit circle $\mathbb{T} = \{z\in\mathbb{C} : \abs{z}=1\}$. The entire function
$M(t) = \int_{\mathbb{T}} e^{t z} \, d\mu(z)$
has no zeros in the open disk $D = \{t\in\mathbb{C} : \abs{t} < \pi/2\}$. Consequently, the Taylor series at $0$ of the analytic branch of $\log M(t)$ has radius of convergence at least $\pi/2$.
\end{theorem}

\begin{proof}
Fix $t \in \mathbb{C}$ with $\abs{t} < \pi/2$. We define the image of the unit circle under the map $z \mapsto e^{tz}$ as the set $\Gamma_t = \{e^{t z} : z\in\mathbb{T}\}$. Since $M(t)$ is the center of mass of $\Gamma_t$ with respect to $\mu$, it lies in the closed convex hull of $\Gamma_t$. To prove $M(t) \neq 0$, it suffices to show that the origin $0$ does not lie in $\mathrm{conv}(\Gamma_t)$.

Let $w = tz$. Since $\abs{z}=1$, we have $\abs{w} = \abs{t} < \pi/2$. Writing $w = u + iv$, the condition $\abs{w} < \pi/2$ implies $\abs{v} < \pi/2$. The real part of a point on the curve is
$\Re(e^w) = e^u \cos(v).$
Since $\abs{v} < \pi/2$, we have $\cos(v) > 0$. Further, $e^u > 0$, so $\Re(e^w) > 0$ for all $w$ such that $\abs{w} < \pi/2$. Geometrically, this means the entire set $\Gamma_t$ is contained strictly within the right half-plane $H = \{ \zeta \in \mathbb{C} : \Re(\zeta) > 0 \}$. The convex hull of a set contained in an open half-plane is also contained in that half-plane. Since $0 \notin H$, it follows that $0 \notin \mathrm{conv}(\Gamma_t)$, and thus $M(t) \neq 0$.

Since $M(0)=1$ and $M$ has no zeros in the simply connected disk $\abs{t}<\pi/2$,
there exists an analytic branch of $\log M(t)$ there. Hence its Taylor series at the origin
has radius of convergence at least $\pi/2$.
\end{proof}

\begin{remark}
\label{rem:info_gap}
This comparison indicates the loss of information inherent in bounding cumulants using only a single absolute moment. The symmetric combinatorial rate $\rho_{\mathrm{sym}} \approx 1.317$ is below the Rademacher benchmark $\pi/2 \approx 1.571$. The gap reflects the loss caused by the two universal relaxations in the proof: the product collapse (\Cref{lem:product_collapse}), which discards the fine structure of lower-order moments, and the triangle inequality, which ignores the algebraic cancellations characterizing cumulants.
\end{remark}

\section{Applications}
\label{sec:applications}

In this section, we illustrate how the centered bound controls standardized cumulants of independent sums.

\subsection{Standardized Cumulants of Independent Sums}
\label{subsec:independent_sums}

A natural statistical use of the centered bound is to control standardized cumulants of independent sums, placing the theory in the natural scale for Edgeworth expansions and higher-order limit theorems.

\begin{proposition}
\label{prop:independent_sums}
Let \(r\ge 2\), and let \(X_1,\dots,X_N\) be independent real-valued random variables such that
\[
\E{\abs{X_i-\E{X_i}}^r} < \infty
\qquad\text{and}\qquad
\Var{X_i} < \infty
\quad \text{for } 1\le i\le N.
\]
Define
\[
S_N \coloneqq \sum_{i=1}^N \bigl(X_i-\E{X_i}\bigr),
\qquad
V_N \coloneqq \sum_{i=1}^N \Var{X_i},
\qquad
Z_N \coloneqq \frac{S_N}{\sqrt{V_N}}.
\]
If \(V_N>0\), then
\[
\abs{\kappa_r(Z_N)}
\le
\Ccen{r}\,V_N^{-r/2}
\sum_{i=1}^N \E{\abs{X_i-\E{X_i}}^r}.
\]
\end{proposition}

\begin{proof}
Let \(Y_i \coloneqq X_i-\E{X_i}\). Then the \(Y_i\) are independent and centered, and
$Z_N = V_N^{-1/2}\sum_{i=1}^N Y_i.$
By additivity of cumulants over independent sums and homogeneity,
$\kappa_r(Z_N)
=
V_N^{-r/2}\sum_{i=1}^N \kappa_r(Y_i).$
Therefore,
\[
\abs{\kappa_r(Z_N)}
\le
V_N^{-r/2}\sum_{i=1}^N \abs{\kappa_r(Y_i)}
\le
\Ccen{r}\,V_N^{-r/2}\sum_{i=1}^N \E{\abs{Y_i}^r},
\]
where the last step follows from \Cref{cor:centered-refinement}.
\end{proof}

\begin{remark}
\label{rem:lyapunov_fraction}
It is convenient to write
$L_{N,r} \coloneqq V_N^{-r/2}\sum_{i=1}^N \E{\abs{X_i-\E{X_i}}^r}.$
Then \Cref{prop:independent_sums} becomes $\abs{\kappa_r(Z_N)} \le \Ccen{r}\,L_{N,r}$. Thus the same Lyapunov fraction that appears in classical limit theorems also controls the standardized \(r\)-th cumulant with an explicit universal constant.
\end{remark}

\begin{corollary}
\label{cor:iid_standardized_cumulants}
Let \(X_1,\dots,X_N\) be i.i.d.\ with mean \(\mu\), variance \(\sigma^2\in(0,\infty)\), and \(\E{\abs{X_1-\mu}^r}<\infty\). Define $Z_N \coloneqq \frac{1}{\sigma\sqrt{N}}\sum_{i=1}^N (X_i-\mu)$. Then
\[
\abs{\kappa_r(Z_N)}
\le
\Ccen{r}\,N^{1-r/2}\,
\frac{\E{\abs{X_1-\mu}^r}}{\sigma^r}.
\]
\end{corollary}

\begin{proof}
This follows immediately from \Cref{prop:independent_sums} with $V_N=N\sigma^2$.
\end{proof}

\begin{corollary}
\label{cor:sample_mean_cumulants}
Under the assumptions of \Cref{cor:iid_standardized_cumulants}, the sample mean $\bar X_N \coloneqq \frac{1}{N}\sum_{i=1}^N X_i$ satisfies
\[
\abs{\kappa_r(\bar X_N)}
\le
\Ccen{r}\,N^{1-r}\,\E{\abs{X_1-\mu}^r}.
\]
\end{corollary}

\begin{proof}
Since
$\bar X_N-\mu = \frac{\sigma}{\sqrt{N}} Z_N,$
homogeneity gives
\[
\kappa_r(\bar X_N-\mu)
=
\left(\frac{\sigma}{\sqrt{N}}\right)^r \kappa_r(Z_N).
\]
Because \(r\ge2\), cumulants are shift invariant, so
$\kappa_r(\bar X_N)=\kappa_r(\bar X_N-\mu).$
Substituting the bound from \Cref{cor:iid_standardized_cumulants} yields the result.
\end{proof}

\begin{remark}
For every fixed \(r\ge 3\), \Cref{cor:iid_standardized_cumulants} yields the explicit decay $\abs{\kappa_r(Z_N)} = O\bigl(N^{1-r/2}\bigr)$. This is the standard scaling governing higher-order standardized cumulants in Edgeworth-type expansions.
\end{remark}

\section{Discussion and Conclusion}
\label{sec:discussion}

Determining explicit bounds for high-order cumulants is a recurring problem in probability theory, particularly when precise non-asymptotic control is required. In this work, we established a unified framework that bounds the $n$-th cumulant $\kappa_n(X)$ solely in terms of an $n$-th absolute moment $M_n(X)$. By identifying the bounding coefficient with the exact combinatorial mass of the partition formula, we derived constants scaling as $(n-1)! / \rho^n$, offering a decisive exponential improvement over the $n^n$ scaling typical of crude combinatorial bounds.

\paragraph{The Combinatorial Dictionary.}
The power of this framework lies in its adaptability. By mapping distributional properties (such as centering or symmetry) to constraints on set partitions (no singletons, or no blocks of odd size), we transformed the estimation problem into the combinatorial problem of enumerating restricted partitions. As summarized in \Cref{tab:summary} and \Cref{tab:explicit_values}, this provides a precise dictionary translating statistical assumptions into quantitative improvements of the exponential rate $\rho$.

\paragraph{The Price of Universality.}
As demonstrated in \Cref{sec:strict}, the bounds are strict whenever the relevant coefficient is positive; the only identity occurs in the centered variance case $n=2$, while in symmetric odd orders the coefficient is zero and the cumulant vanishes identically. The comparison with the Rademacher law shows that the sharpest symmetric rate in this paper remains below the explicit benchmark $\pi/2$ arising from $\log(\cosh t)$. This discrepancy reflects the loss of information caused by the two universal relaxations in our proof: by applying the triangle inequality and a uniform product collapse (\Cref{lem:product_collapse}), we discard the fine structure of lower-order moments and the algebraic cancellations native to cumulants. This suggests that improving the exponential rate beyond $\rho_{\mathrm{sym}}$ would require additional information beyond a single absolute moment, such as explicit control of lower-order moments or cancellations within the partition formula.

\paragraph{Concluding Remarks.}
Ultimately, these universal bounds sharpen the most general moment-based control of cumulants without imposing structural or analytic assumptions beyond finiteness of the relevant moment. Because they depend only on a single absolute moment of the same order, they remain applicable in settings where moment generating functions may fail to exist. By providing concrete, readily computable constants, the framework offers a useful tool for fixed-order non-asymptotic analysis and for quantitative questions involving standardized cumulants and limit theorems.

\bibliographystyle{alpha}
\bibliography{ref}

@incollection{doring2022method,
  title={The method of cumulants for the normal approximation},
  author={D{\"o}ring, Hanna and Eichelsbacher, Peter},
  booktitle={Modern Problems of Stochastic Analysis and Statistics. SPASS 2021},
  pages={151--186},
  year={2023},
  publisher={Springer}
}

@book{ProkhorovRozanov1969_en,
  author    = {Prokhorov, Yuriy V. and Rozanov, Yuriy A.},
  title     = {Probability Theory: Basic Concepts, Limit Theorems, Random Processes},
  publisher = {Springer-Verlag},
  address   = {Berlin, Heidelberg},
  year      = {1969},
  series    = {Die Grundlehren der mathematischen Wissenschaften in Einzeldarstellungen},
  volume    = {157},
  doi       = {10.1007/978-3-642-88076-9}
}

@article{dubkov1976properties,
  title={Properties and interdependence of the cumulants of a random variable},
  author={Dubkov, AA and Malakhov, AN},
  journal={Radiophysics and Quantum Electronics},
  volume={19},
  number={8},
  pages={833--839},
  year={1976},
  publisher={Springer}
}

@misc{pinelis2020bounds,
  author       = {Pinelis, Iosif},
  title        = {Bounds on cumulants in terms of moments},
  howpublished = {MathOverflow},
  year         = {2020},
  month        = {Aug},
  note         = {URL: \url{https://mathoverflow.net/q/368761} (version: 2020-08-10)},
  urldate      = {2025-09-16}
}

@misc{jones2023bound,
  author       = {Jones, Chris},
  title        = {Bound for cumulants of bounded random variables},
  howpublished = {Math.StackExchange},
  year         = {2023},
  month        = {Oct},
  note         = {URL: \url{https://math.stackexchange.com/q/3135549} (version: 2023-10-27)},
  urldate      = {2025-09-16}
}

@article{speed1983cumulants,
  title={Cumulants and partition lattices},
  author={Speed, T. P.},
  journal={Australian \& New Zealand Journal of Statistics},
  volume={25},
  number={2},
  pages={378--388},
  year={1983},
  publisher={Wiley Online Library}
}

@book{mccullagh1987tensor,
  title={Tensor Methods in Statistics},
  author={McCullagh, Peter},
  year={1987},
  publisher={Chapman and Hall/CRC},
  series={Monographs on Statistics and Applied Probability}
}

@article{statulevivcius1966large,
  title={On large deviations},
  author={Statulevi{\v{c}}ius, VA},
  journal={Zeitschrift f{\"u}r Wahrscheinlichkeitstheorie und verwandte Gebiete},
  volume={6},
  number={2},
  pages={133--144},
  year={1966},
  publisher={Springer}
}

@article{leonov1959method,
  author    = {Leonov, V. P. and Shiryaev, A. N.},
  title     = {On a method of calculation of semi-invariants},
  journal   = {Theory of Probability and its Applications},
  volume    = {4},
  number    = {3},
  pages     = {319--329},
  year      = {1959}
}

@book{stoyanov2013counterexamples,
  title={Counterexamples in probability},
  author={Stoyanov, Jordan M},
  year={2013},
  publisher={Dover Publications}
}

@article{kleiber2013multivariate,
  title={Multivariate distributions and the moment problem},
  author={Kleiber, Christian and Stoyanov, Jordan},
  journal={Journal of Multivariate Analysis},
  volume={113},
  pages={7--18},
  year={2013},
  publisher={Elsevier}
}

@article{lin2002moment,
  title={On the moment determinacy of the distributions of compound geometric sums},
  author={Lin, Gwo Dong and Stoyanov, Jordan},
  journal={Journal of Applied probability},
  volume={39},
  number={3},
  pages={545--554},
  year={2002},
  publisher={Cambridge University Press}
}

@article{pinelis2015rosenthal,
author = {Iosif Pinelis},
title = {{Exact Rosenthal-type bounds}},
volume = {43},
journal = {The Annals of Probability},
number = {5},
publisher = {Institute of Mathematical Statistics},
pages = {2511 -- 2544},
year = {2015},
doi = {10.1214/14-AOP942},
URL = {https://doi.org/10.1214/14-AOP942}
}

@article{pinelis1994optimum,
  title={Optimum bounds for the distributions of martingales in Banach spaces},
  author={Pinelis, Iosif},
  journal={The Annals of Probability},
  pages={1679--1706},
  year={1994},
  publisher={JSTOR}
}

@article{pinelis2007asymmetric,
  title={Exact inequalities for sums of asymmetric random variables, with applications},
  author={Pinelis, Iosif},
  journal={Probability Theory and Related Fields},
  volume={139},
  number={3},
  pages={605--635},
  year={2007},
  publisher={Springer}
}

@article{stoyanov2020nonconventional,
  title={Nonconventional limits of random sequences related to partitions of integers},
  author={Stoyanov, Jordan and Vignat, Christophe},
  journal={Proceedings of the American Mathematical Society},
  volume={148},
  number={4},
  pages={1791--1804},
  year={2020}
}

@article{stoyanov2020new,
  title={New checkable conditions for moment determinacy of probability distributions},
  author={Stoyanov, Jordan M and Lin, Gwo Dong and Kopanov, Peter},
  journal={Theory of Probability \& Its Applications},
  volume={65},
  number={3},
  pages={497--509},
  year={2020},
  publisher={SIAM}
}

@article{utev1989cumulants,
  title={Cumulants and moment inequalities},
  author={Utev, Sergei Aleksandrovich},
  journal={Theory of Probability \& Its Applications},
  volume={34},
  number={4},
  pages={742--747},
  year={1989},
  publisher={SIAM}
}

@book{flajolet2009analytic,
  title={Analytic Combinatorics},
  author={Flajolet, Philippe and Sedgewick, Robert},
  year={2009},
  publisher={Cambridge University Press}
}

\end{document}